\documentclass[12pt,a4paper]{article} 
\usepackage{amsmath,amssymb,amsthm,amsfonts,amscd,euscript,verbatim, t1enc, newlfont}
\usepackage{hyperref}
 \theoremstyle{plain}
\newtheorem{theo}{Theorem}[subsection]

\newtheorem{pr}[theo]{Proposition}

 \newtheorem{lem}[theo]{Lemma}

\theoremstyle{remark}
\newtheorem{rema}[theo]{Remark}

\theoremstyle{definition}
\newtheorem{defi}[theo]{Definition}
\newtheorem*{notat}{Notation}

 \newcommand\lan{\langle}
\newcommand\ra{\rangle}
\newcommand\bl{\bigl(} \newcommand\br{\bigl)}

\newcommand\ob{^{-1}}

\newcommand\dmge{DM^{eff}_{gm}{}}

\newcommand\dmgm{DM_{gm}}

\newcommand\obj{Obj}
\newcommand\mo{Mor}
\newcommand\id{\operatorname{id}}
\newcommand\cu{\underline{C}}
\newcommand\du{{\underline{D}}}
\newcommand\eu{{\underline{E}}}
\newcommand\au{\underline{A}}

\newcommand\hw{{\underline{Hw}}}

\newcommand\z{{\mathbb{Z}}}

\newcommand\q{{\mathbb{Q}}}

\newcommand\p{\mathbb{P}}

\newcommand\al{\alpha}
\newcommand\be{\beta}

\newcommand\ns{\{0\}}

\DeclareMathOperator\prli{\varprojlim}
\DeclareMathOperator\inli{\varinjlim}

\newcommand\chow{Chow}

\newcommand\chowe{Chow^{eff}}

\newcommand\spe{\operatorname{Spec}\,}

\DeclareMathOperator\imm{\operatorname{Im}}
\DeclareMathOperator\co{\operatorname{Cone}}

\DeclareMathOperator\cha{\operatorname{char}}

\newcommand\dms{DM(S)}
\newcommand\dmcs{DM_c(S)}

\newcommand\dmk{DM(K)}

\newcommand\dmlas{DM_{\Lambda}(S)}

\newcommand\dmqs{DM_{\q}(S)}
\newcommand\dmcqs{DM_{c,\q}(S)}
\newcommand\dmq{DM_{\q}}

\newcommand\dmx{DM(X)}
\newcommand\dmcx{DM_c(X)}
\newcommand\dm{DM}

\newcommand\dmc{DM_c}
\newcommand\dmy{DM(Y)}

\newcommand\hwchows{{\underline{Hw}}_{\chow(S)}}

\newcommand\wchow{{w_{Chow}}}

\newcommand\wchows{{w_{Chow}(S)}}

\newcommand\sss{{\mathcal{S}}}

\begin{document}

\title{
On weights for relative motives with integral coefficients} 
 \author{Mikhail V. Bondarko
   \thanks{ 
 The work is supported by RFBR
(grants no. 12-01-33057 and 14-01-00393) and by the Saint-Petersburg State University research grant no. 6.38.75.2011.} }
  \maketitle
  

\begin{abstract} 
The goal of this paper is to define a certain {\it Chow weight structure} $\wchow$ for the category $\dms$ of 
Voevodsky's motivic complexes with integral coefficients (as described by Cisinski and Deglise) over any 
 excellent finite-dimensional separated  
 scheme $S$. Our results are parallel to (though substantially weaker than) the corresponding 'rational coefficient' statements proved by D. Hebert and the author. 
 
 As an immediate consequence of the existence of $\wchow$, we obtain certain (Chow)-weight spectral sequences and filtrations for any (co)homology of $S$-motives.
\end{abstract}  
\tableofcontents

 \section*{Introduction}
The goal of  this paper is to prove 
  that the {\it Chow weight structure} $\wchow$ (as introduced in \cite{bws} and in \cite{bzp} for  Voevodsky's motives over a perfect field $k$) can also be defined (somehow) for the category $\dms$ of motives with integral coefficients over any excellent  separated finite dimensional base scheme $S$. 

As was shown in \cite{bws}, the existence of $\wchow$ yields several nice consequences.
In particular, there exist    {\it Chow-weight} spectral sequences and  filtrations 
 for any (co)homological functor $H:\dms\to \au$.
 

Now we list the contents of the paper. More details can be found at the beginnings of sections.

 In \S\ref{sprem} we recall some basic properties of countable homotopy colimits in triangulated categories,  Voevodsky's motives, and weight structures. Most of the statements of the section are contained in \cite{neebook}, \cite{degcis}, and \cite{bws}; yet we also prove some new results.

 In \S\ref{swchow} we define $\wchow$ and study its properties. In particular, we study  the 'functoriality' of $\wchow$ (with respect to the functors of the type $f^*,f_*,f^!,f_!$, for $f$ being a quasi-projective morphism of schemes).
 More 'nice' properties of $\wchows$ can be verified 
 for $S$ being a regular scheme over a field. 
 We also describe the relation of $\wchow(S)$ with its 'rational' analogue for $\dmqs$.

 The author is deeply grateful to prof. F. Deglise for his very helpful explanations.

\begin{notat}

 For a category $C,\ A,B\in\obj C$, we denote by
$C(A,B)$ the set of  $C$-morphisms from  $A$ into $B$.

For categories $C,D$ we write $C\subset D$ if $C$ is a full 
subcategory of $D$.

For a category $C,\ X,Y\in\obj C$, we say that $X$ is a {\it
retract} of $Y$ if $\id_X$ can be factorized through $Y$
(if $C$ is triangulated or abelian, then $X$ is a  retract of $Y$
whenever $X$ is its direct summand).
The Karoubization of $B$ is
the category of 'formal images' of idempotents in $B$
(so $B$ is embedded into an idempotent complete category).

 For an additive $D\subset C$ the subcategory $D$ is called
{\it Karoubi-closed}
  in $C$ if it
contains all retracts  of its objects in $C$. The full subcategory of $C$ whose objects are all retracts of objects of $D$ (in $C$) will be called the {\it Karoubi-closure} of $D$ in $C$.

$M\in \obj C$ will be called compact if the functor $C(M,-)$
commutes with all small coproducts that exist in $C$. In this paper (in contrast with \cite{bws}) we will only consider compact objects in those categories that are closed with respect to arbitrary small coproducts. 

$\cu$ below will always denote some triangulated category; usually it will
be endowed with a weight structure $w$ (see Definition \ref{dwstr}
below).

We will use the term 'exact functor' for a functor of
triangulated categories (i.e.  for a functor that preserves the
structures of triangulated categories). 
We will call a covariant (resp. contravariant)
additive functor $H:\cu\to \au$ for an abelian $\au$
 {\it homological} (resp. {\it cohomological}) if
it converts distinguished triangles into long exact sequences.

For $f\in\cu (X,Y)$, $X,Y\in\obj\cu$, we will call the third vertex
of (any) distinguished triangle $X\stackrel{f}{\to}Y\to Z$ a cone of
$f$; recall that distinct choices of cones are connected by
(non-unique) isomorphisms.
We will often specify a distinguished triangle by two of its
morphisms.

 For a set of
objects $C_i\in\obj\cu$, $i\in I$, we will denote by $\lan C_i\ra$
the smallest strictly full triangulated subcategory containing all $C_i$; for
$D\subset \cu$ we will write $\lan D\ra$ instead of $\lan \obj
D\ra$. We will call the  Karoubi-closure of $\lan C_i\ra$ in $\cu$ the {\it triangulated category  generated by $C_i$}.

For $X,Y\in \obj \cu$ we will write $X\perp Y$ if $\cu(X,Y)=\ns$.
For $D,E\subset \obj \cu$ we will write $D\perp E$ if $X\perp Y$
 for all $X\in D,\ Y\in E$.
For $D\subset \cu$ we will denote by $D^\perp$ the class
$$\{Y\in \obj \cu:\ X\perp Y\ \forall X\in D\}.$$
Sometimes we will denote by $D^\perp$ the corresponding
 full subcategory of $\cu$. Dually, ${}^\perp{}D$ is the class
$\{Y\in \obj \cu:\ Y\perp X\ \forall X\in D\}$.

We will say that some $C_i$, $i\in I$, {\it weakly generate} $\cu$ if for
$X\in\obj\cu$ we have: $\cu(C_i[j],X)=\ns\ \forall i\in I,\
j\in\z\implies X=0$ (i.e. if $\{C_i[j]\}^\perp$ contains only
 zero objects).

Below by default all schemes will be
excellent separated 
of finite Krull dimension. We will call a morphism $X\to S$ quasi-projective if it factorizes through an embedding of 
$X$ into $\p^N(S)$ (for a large enough $N$).
We will call a scheme (satisfying these conditions) {\it pro-smooth} if it can be presented as the limit 
 of a filtering projective system of smooth 
 varieties over a field $K$ such that the connecting morphisms are smooth, 
  affine, and dominant.
  Note that here we can assume $K$ to be perfect, since for any field its spectrum is the limit of a 'smooth affine dominant' projective system 
  of smooth 
 varieties over the corresponding prime field. 
Besides, any pro-smooth scheme is regular since 
 a filtered inductive limit of regular rings is regular if it is Noetherian.

We will sometimes need certain stratifications of a (regular) scheme $S$. Recall that  a 
stratification 
 $\al$  is a presentation of  $S$ as $\cup S_l^\al$
 where $S_l^\al$, $l\in L$ ($L$ is a finite set), are pairwise disjoint locally closed subschemes of $S$. Omitting $\al$, we will denote by
$j_l:S_l^\al\to S$  the corresponding immersions.
 We do not demand the closure of each $S_l^\al$ to be  the union of strata; 
 we will only assume that
 $S$ 'can be glued from $S_l^\al$ step by step'. This means: there exists a 
 rooted
 binary tree $T^{\al}$ whose 
 leaf nodes are exactly $S_l^\al$, any parent node  $X$ of $T^{\al}$ has exactly two child nodes, 
 such that 
 the union of the leaf node 
 descendants for the left child node of $X$ is open in the
  union of all the leaf node
 descendants of $X$ (see \url{http://en.wikipedia.org/wiki/Binary_tree} for the corresponding definitions).
 We will say that a stratification $\al$ is the union of stratifications $\beta$ and $\gamma$ if the roots of the trees for 
 $\beta$ and $\gamma$ are the child nodes for the root for $\al$.

 Below we will only consider  stratifications that we call {\it very regular}; this means that there exists a tree $T^\al$ as above such that  for any  node  $X$ of $T^{\al}$ the
  union of the leaf node
 descendants of $X$ (this includes $X$  if $X$ is a leaf node itself) is regular.
 
Below we will identify a Zariski point (of a scheme $S$) with the spectrum of its residue field.

We will call a morphism $f:Y\to H$ quasi-projective if there exists a quasi-compact immersion $Y\to \p^N(X)$. 

 \end{notat}

\section{Preliminaries: relative motives and weight structures}\label{sprem}


In \S\ref{scolim} we recall the theory of countable (filtered)
homotopy colimits and study certain related notions.

In \S\ref{sbrmot}  we recall some of basic properties of motives over $S$ (as considered in \cite{degcis}; we also deduce certain results that were not stated in ibid. explicitly).

In \S\ref{sws} we recall some basics of the theory of weight structures (as developed in \cite{bws}); we also prove some new lemmas on the subject.

\subsection{Countable homotopy colimits in triangulated categories and big extension-closures
 }\label{scolim}

 We recall the basics of the theory  of countable (filtered)
homotopy colimits in triangulated categories (as introduced in \cite{bockne}; some more details can be found in \cite{neebook} and in \S4.2 of \cite{bws}). 
We will only apply the results of this paragraph to triangulated categories closed
with respect to arbitrary small coproducts; so we will not mention this restriction below (though this does not really decrease the generality of our results).

\begin{defi}\label{dcoulim}

Suppose that we have a sequence of objects $Y_i$ (starting from some
$j\in\z$)  and maps $\phi_i:Y_{i}\to Y_{i+1}$.  
We consider the map $d:\oplus
\id_{Y_i}\bigoplus \oplus (-\phi_i): D\to D$ (we can define it since
its $i$-th component 
 can be easily factorized  as the composition
$Y_i\to Y_i\bigoplus Y_{i+1}\to D$).
 Denote  a cone of $d$ by $Y$. We will write $Y=\inli Y_i$ and
 call $Y$ the {\it homotopy colimit} of $Y_i$; we will not consider any other (homotopy) colimits in this paper.

\end{defi}

\begin{rema}\label{rcoulim}

1. Note that these homotopy colimits are not really canonical and functorial in $Y_i$ since the choice of a cone is not canonical. They are only defined up to non-canonical isomorphisms.

2. By Lemma 1.7.1 of \cite{neebook} the homotopy colimit of
$Y_{i_j}$ is the same for any subsequence of $Y_i$. 
In particular,
we can discard any (finite) number of first terms in $Y_i$.

3. By Lemma 1.6.6 of \cite{neebook} the homotopy colimit of
$X\stackrel{\id_X}{\to}X\stackrel{\id_X}{\to}
X\stackrel{\id_X}{\to} X\stackrel{\id_X}{\to}\dots$ is $X$. Hence
we obtain that $\inli X_i\cong X$ if for $i\gg 0$ all $\phi_i$ are
isomorphisms and $X_i\cong X$.

\end{rema}

We also recall the behaviour of colimits  under (co)representable
functors.

\begin{lem}\label{coulim} 
1.  For any $C\in\obj\cu$ we have a natural surjection $
\cu(Y,C)\to \prli \cu(Y_i,C)$.

2. This map is bijective if all $\phi_i[1]^*:
\cu(Y_{i+1}[1],C)\to \cu(Y_i[1],C)$  are surjective for all $i\gg
0$.

3. If $C$ is compact then $\cu(C,Y)= \inli \cu(C,Y_i)$.
\end{lem}

Below we will also need a new (simple) piece of homological algebra: the  definition of a {\it strongly extension-closed} class of objects, and some properties of this notion.


\begin{defi}\label{dses}

1. $D\subset \obj \cu$ will be
called {\it extension-closed}
    if it contains $0$ and for any distinguished triangle $A\to B\to C$
in $\cu$ we have: $A,C\in D\implies
B\in D$ (hence it is also {\it strict}, i.e. contains all objects of $C$ isomorphic to its elements). We will call the smallest extension-closed subclass of
 objects of $\cu$ that  contains a class $C'\subset \obj\cu$    the 
extension-closure of $C'$.

 2. $C\subset \obj \cu$ will be called {\it strongly extension-closed} if it 
 contains $0$, and if for any $\phi_i:Y_{i}\to Y_{i+1}$ such that $Y_0\in C$, $\co(\phi_i)\in C$ for all $i\ge 0$, we have $\inli_{i\ge 0} Y_i\in C$ (i.e. $C$ contains all possible cones of the corresponding distinguished triangle; note that those are isomorphic).

3. The smallest strongly extension-closed Karoubi-closed class of objects of $\cu$ that  contains a class $C'\subset \obj\cu$ and is closed with respect to arbitrary (small) coproducts will be called the 
big extension-closure of $C'$.

\end{defi}

Now we verify certain properties of these notions. 

\begin{lem}\label{lbes}
\begin{enumerate}

\item\label{iseses}
 Let $C$ be strongly extension-closed. Then it is also extension-closed. 

\item\label{isesperp}
Suppose that for $C',D\subset \obj \cu$, 
we have 
$ C'\perp D$. Then for the 
 big extension-closure $C$ of $C'$ we also have $C\perp D$.

\item\label{isescperp}
Suppose that for $C',D\subset \obj \cu$, all objects of $D$ are compact, we have 
$D\perp C'$. Then $D$ is also orthogonal to the 
big extension-closure  $C$ of $C'$. 


\item\label{isefun} For an exact functor $F:\cu \to \du$, any $C'\subset \obj\cu$, and its extension-closure $C$ we have: 
the class $F(C)$ is contained in the
 extension-closure  of $F(C')$ in $\du$.

\item\label{isesfun} Suppose that an exact functor $F:\cu \to \du$ of triangulated categories commutes with 
arbitrary coproducts. Then for any $C'\subset \obj\cu$ and its big extension-closure $C$ 
the class $F(C)$ is contained in the
big extension-closure  of $F(C')$ in $\du$.

\end{enumerate}
\end{lem}
\begin{proof}

\ref{iseses}. It suffices to note for any  distinguished triangle $X\to Y\to Z$ the object $Y$ is the colimit of $X\stackrel{f}{\to} Y\stackrel{\id_Y}{\to} Y \stackrel{\id_Y}{\to} Y \stackrel{\id_Y}{\to} Y\to \dots$; the cone of $f$ is $Z$, whereas the cone of $\id_Y$ is $0$.

\ref{isesperp}. Since for any $d\in D$ the functor $\cu(-,d)$ converts arbitrary coproducts into products,
it suffices to verify  (for any $d\in D$): if for $Y_i$ as in
 Definition \ref{dses}(2) we have $Y_0\perp d$, $\co(\phi_i)\perp d$ for all $i\ge 0$, then $\inli Y_i\perp d$.
 Now, for any $i\ge 0$ we have a long exact sequence
 $$\dots  \to \cu(Y_{i+1}[1],d) \to \cu(Y_i[1],d) \to \cu(\co(\phi_i),d) (=\ns) \to \cu(Y_{i+1},d) \to \cu(Y_i,d)\to \dots  $$
 Hence $\cu(Y_{i+1}[1],d)$ surjects onto $  \cu(Y_i[1],d)$, whereas the obvious induction yields that $\cu(Y_j,d)=\ns$ for any $j\ge 0$. Then Lemma \ref{coulim}(2) yields  $\cu(\inli Y_i,d)\cong \prli \cu(Y_i,d)=\ns$. 

\ref{isescperp}. Note that for any $d\in D$ the functor $\cu(d,-)$ converts arbitrary coproducts into coproducts. So, similarly to the reasoning above, it suffices to  fix some $d\in D$ and prove that $d\perp \inli_{i\ge 0}Y_i$ if  
$d\perp Y_0$ and $d\perp \co(\phi_i)$ for all $i\ge 0$. The half-exact sequence $\cu(d,Y_i)\to \cu(d,Y_{i+1})\to \cu(d, \co(\phi_i))$ yields that $d\perp Y_j$ for all $j\ge 0$. It remains to apply Lemma \ref{coulim}(3).

\ref{isefun}, \ref{isesfun}. Obvious from the definitions given above.

\end{proof}

\begin{rema}\label{rperp}
Assertion \ref{isescperp} of the Lemma immediately implies the following simple fact: for any class $D$ of compact objects of $\cu$ the class $D^{\perp}$ is strongly extension-closed, Karoubi-closed, and also closed with respect to arbitrary coproducts.
Indeed, it suffices to note that $D$ is orthogonal to the big extension-closure of $D^{\perp}$. 
\end{rema}

\subsection{On relative Voevodsky's  motives with integral coefficients 
(after Cisinski and Deglise)}\label{sbrmot}

We list some   properties of the triangulated categories  $DM(-)$ of relative motives
(as considered by Cisinski and Deglise).

\begin{pr}\label{pcisdeg}

Let $X,Y$ be  (excellent separated finite dimensional) 
schemes; $f:X\to Y$ is a  
finite type morphism. 
\begin{enumerate}

\item\label{imotcat} For any  $X$ a tensor triangulated  category $\dmx$ with the unit object $\z_X$ is defined; it is closed with respect to arbitrary small coproducts.

$\dmx$ is the category of  {\it Voevodsky's  motivic complexes} over $X$, as described (and thoroughly studied) in 
\S11 of \cite{degcis}.

\item\label{iadd} $\dm(X\sqcup Y)=\dm(X)\bigoplus \dm(Y)$; $\dm(\emptyset)=\ns$.

\item\label{imotgen}
The (full) subcategory $\dmcx\subset \dmx$ of compact objects is tensor triangulated, and  $\q_X\in \obj \dmcs$. $\dmcx$ weakly generates $\dmx$.

\item\label{imotfun}  For any  $f$ of finite type
the following functors
 are defined:
$f^*: \dm(Y) \leftrightarrows \dmx:f_*$ and $f_!: \dmx \leftrightarrows \dmy:f^!$; $f^*$ is left adjoint to $f_*$ and $f_!$ is left adjoint to $f^!$.

We call these the {\bf motivic image functors}.
Any of them (when $f$ varies) yields a  $2$-functor from the category of 
(separated finite-dimensional excellent) schemes
with  morphisms of finite type to the $2$-category of triangulated categories.

Besides,  the functors $f^*$ and $f_!$
 preserve compact objects (i.e. they could be restricted to the subcategories $\dmc(-)$) and arbitrary (small) coproducts in $\dm(-)$.

\item\label{iexch} 
For a Cartesian square
of finite type 
morphisms
$$\begin{CD}
Y'@>{f'}>>X'\\
@VV{g'}V@VV{g}V \\
Y@>{f}>>X
\end{CD}$$
we have $g^*f_!\cong f'_!g'{}^*$ and $g'_*f'{}^!\cong f^!g_*$.
if $g$ is smooth or if $f$ is smooth projective.

\item\label{itate}  For any $X$ there exists a
Tate object $\z(1)\in\obj\dmcx$; tensoring by it yields  exact Tate twist functors $-(1)$ on $\dmcx\subset \dmx$.
Both of these  functors are auto-equivalences; we will denote the corresponding inverse functors by $-(-1)$.

 Tate twists
commute with all motivic image functors mentioned (up to an isomorphism of functors) and with arbitrary (small) direct sums. 

Besides, for $X=\p^1(Y)$ there is a functorial isomorphism $f_!(\q_{\z^1(Y)})\cong \z_Y\bigoplus \z_Y(-1)[-2]$. 

\item\label{icompgen} $\dmcx$ is the triangulated subcategory of $\dm(X)$ generated by all $u_!\z_U(n)$ for all 
$u$ being the compositions of  finite \'etale morphisms with  open embeddings and  smooth projective morphisms, $n\in \z$. 

\item\label{iupstar}  $f^*$ is symmetric monoidal; $f^*(\z_Y)=\z_X$.

\item \label{ipur}

$f_*\cong f_!$ if $f$ is proper;
$f^!(-)\cong f^*(-)(s)[2s]$ 
 if $f$ is smooth 
 quasi-projective 
 (everywhere) of relative dimension $s$. 


\item\label{iglup}
Let $i:Z\to X$  be a closed immersion, $U=X\setminus Z$; $j:U\to X$ is the complementary open immersion. Then 
the compositions $i_*j^*$, $i^*j_!$, and $i^!j_*$ are zero, whereas the adjunction transformation $j^*j_*\to 1_{\dm(U)}$ is an isomorphisms of
      functors.


\item\label{iglu}

In addition to the assumptions of the previous assertion let $Z,X$ be regular.
Then
the motivic image functors yield {\it gluing data} for $\dm(-)$  (in the sense of \S1.4.3 of \cite{bbd}; see also Definition 8.2.1 of \cite{bws}). That means that 
(in addition to the statements given by the previous assertions) the following statements are also valid.

(i)  $i_*$ is a full embeddings; $j^*=j^!$ is isomorphic to the
localization (functor) of $\dmx$ by
$i_*(\dm(Z))$. Hence the adjunction transformations $i^*i_*\to 1_{\dm(Z)}\to
      i^!i_!$ and $1_{\dm(U)}\to j^!j_!$ 
       are isomorphisms of
      functors.

(ii) For any $M\in \obj \dmx$ the pairs of morphisms 
\begin{equation}\label{eglu}
j_!j^!(M)(=j_!j^*(M)) \to M
\to i_*i^*(M)(\cong i_!i^*M)\end{equation}
 and $i_!i^!(M) \to M \to j_*j^*(M)$ can be completed to
distinguished triangles (here the connecting
morphisms come from the adjunctions of assertion \ref{imotfun}).


\item\label{icont}
Let $S$ be a scheme which is the limit of an essentially affine (filtering) projective system of  schemes $S_\be$ (for $\be\in B$) such that the connecting morphisms are dominant. Then $\dmcs$ is isomorphic to the $2$-colimit 
 of the categories $\dmc(S_\be)$. For these isomorphisms all the connecting functors are given by the corresponding motivic inverse image functors; see the next assertion  and  Remark \ref{ridmot} below.

\item\label{icontp} In the setting of the previous assertion for some $\be_0\in B$ we denote the corresponding morphisms $S\to S_{\be_0}$ and $S_{\be}\to S_{\be_0}$ (when the latter is defined) by $p_{\be_0}$ and $p_{\be,\be_0}$, respectively. Let $M\in \dmc(S_{\be_0})$, $N\in \dm(S_{\be_0})$. 

Then $\dms(p_{\be_0}^*M, p_{\be_0}^*N)=\inli_{\be} \dm(S_{\be})(p_{\be,\be_0}^*M, p_{\be,\be_0}^*N)$.

\end{enumerate}

\end{pr}
\begin{proof}
Almost all of these properties of 
Voevodsky's motivic complexes
are stated in (part B of) the Introduction of ibid.; the proofs are mostly contained in 
\S11 and \S10 of ibid. In particular, see Theorem 11.4.5 of loc. cit.; note that a quasi-projective $f$ mentioned in assertion \ref{ipur} can be presented as the composition of an open embedding with a smooth projective morphism. The functors preserving compact motives are treated in \S4.2 of ibid. Our assertions \ref{iglup} and \ref{iglu} were proved in  \S2.3 of ibid.
Assertion \ref{icontp} is given by 
formula (4.3.4.1) of ibid.
$f^*$ and $f_!$ respect coproducts since they admit right adjoints.

So, we will only prove those assertions that are not stated in ibid. (explicitly).


Assertion \ref{icompgen} is an immediate consequence of Theorem 11.1.13 of \cite{degcis}. One only should unwind the definitions involved. 
In loc. cit. the generators were given by Tate twists of the motives $M_X(U)$ for all smooth $u$. Now,
 the ('Nisnevich') Mayer-Vietoris distinguished triangle (see \S11.1.9(1) of loc. cit.) allows us to consider only 
generators 
$U$ as above (and also connected); in this case we can replace $M_X(U)$ by a Tate twist of $u_*\z_U$.

\end{proof}

\begin{rema}\label{ridmot}

 



 In \cite{degcis} the functor $g^*$
  was constructed for any 
   morphism $g$ not necessarily of finite type; it commutes with arbitrary coproducts and   preserves compact objects. 
   Besides, for such a $g$ and any 
   smooth projective $f:X'\to X$ we have an isomorphism $g^*f_*\cong f'_*g'{}^*$ (for the corresponding $f'$ and $g'$; cf. parts \ref{iexch} and \ref{ipur} of the proposition).


 Note also: if $g$ is a pro-open immersion, then one can also define $g^!=g^*$.
 So, one can also define $j_K^!$ that commutes with arbitrary coproducts. 
 The system of these functors satisfy the second assertion in part \ref{iexch} of the proposition (for a finite type 
  $g$).

4. In \cite{degcis} properties of motives with rational coefficients (those were called Beilinson motives) were studied in great detail (see Appendix C of ibid., it is no wonder that it is easier to deal with rational coefficients than with integral ones). 


\end{rema}

When treating pro-smooth schemes, we will need the following statement.

\begin{lem}\label{lega43}
Let $X=\inli_{\be \in B} X_\be$,  be pro-smooth, $Y$ is a regular scheme quasi-projective over $X$. Then 
there exist a $\be_0\in B$ and a   
 quasi-projective $Y_{\be_0}/X_{\be_0}$ such that:
$Y\cong Y_{\be_0}\times_{X_{\be_0}}X$ and for all $\be\ge \be_0$ the schemes $Y_\be =Y_{\be_0}\times_{X_{\be_0}}X_\be$ are regular. 


\end{lem}
\begin{proof}
Let $Y$ be closed in an open subscheme $U$ of $\p^n(X)$ (for some $n>0$). By Proposition 8.6.3 of \cite{ega43}, there exists a $\be_1\in B$, an open $U_{\be_1}\subset X_{\be_1}$, and a closed $Y_{\be_1}\subset U_{\be_1}$ such that $Y=Y_{\be_1}\times_{X_{\be_1}}X$. Hence it suffices to consider the case when $Y=Y_{\be_1}\times_{X_{\be_1}}X$, $Y_{\be_1}$ is closed in $X_{\be_1}$.
 
 The question is whether (in the situation described) we can choose 
 a $\be_0\ge \be_1$ such that all the corresponding $Y_\be$ are regular.
 Since the connecting morphisms of $Y_\be$ are smooth, it suffices to find a $\be_0$ such that  $Y_{\be_0}$ is regular.
 Moreover, we obtain that the preimages of the singularity loci  $S_\be$ of $Y_\be$ with respect to the projections $p_\be:Y\to Y_\be$ form a projective system of closed subschemes of $Y$ (i.e. that they 'decrease'). Since $Y$ is noetherian, it suffices to verify: there cannot exist a Zariski point $P$ of $Y$ that  belongs to all of $p_\be\ob (S_\be)$.
 
 Assume the converse. Aapplying loc. cit. again we obtain that  $P$ is the preimage of a Zariski point  $P_0\in Y_{\be_0}$ for  a certain $\be_0\ge \be_1$. Then we obtain that the local equations that characterize $Y_{\be_0}$ in $X_{\be_0}$ at the point $P_0$ do not yield a regular sequence; hence the same is true for $Y$ in $X$ at the point $P$. Thus $Y$ is not regular at $P$, and we obtain a contradiction.

\end{proof}

The following statements 
follow from Proposition \ref{pcisdeg} easily. 
The limit argument that we use in the proof of assertion 2 of the Lemma is closely related with Remark 6.6 of \cite{lesm}.

\begin{lem}\label{l4onepo}

1. 
Let $S=\cup S_l^\al$, $l\in L$,  be a (very regular) stratification (see the Notation section). 

Then  any $M\in \obj \dms$ belongs to the extension-closure (see Definition \ref{dses}(1)) of $j_{l!}j_l^*(M)$.

 In particular, $\z_S$ belongs to the extension-closure  of $j_!(\z_{S_l^\al})$.


2. 
For any pro-smooth $S$ (see the Notation) 
 and a
smooth quasi-projective $x:X\to S$ 
 we have: $\dms(x_!\z_X[-r](-q),\z_S) = \ns$
 if $r>2q$.

\end{lem}

\begin{proof}

1. We prove the assertion by induction on the number of strata. In the case when $\#L=1$ all the assertions are obvious.
 
 Now let $\#L>1$. By definition,   $\al$ is the union of two regular stratifications of $\beta$ and $\gamma$ of $U,Z\subset X$ respectively, such that $X,U,Z$ are regular, $Z$ is closed in $X$.
 We denote  the open immersion $U \to S$ by $j$ and the (closed) immersion $Z\to S$ by $i$.

By the inductive assertion, both $j_!j^*(M)$ and $i_!i^*M$ belong to the envelope in question (here we use the $2$-functoriality of $(-)^*$ and $(-)_!$). Hence the distinguished triangle (\ref{eglu}) yields the result.

2. Certainly, we can assume that $X$ is connected (see Proposition \ref{pcisdeg}(\ref{iadj})).
 We present $x$ as the composition $X\stackrel{i}{\to}Y \stackrel{u}{\to} P\stackrel{p}{\to} S$ where $i$ is a closed embedding of pro-smooth schemes, $u$ is an open embedding, and $p$ is a smooth projective morphism of codimension $d$ with connected domain. Applying 
 Proposition \ref{pcisdeg}(\ref{imotfun}, \ref{ipur},\ref{iupstar}) we obtain:  $$\begin{gathered} \dms(x_!\z_X[-r](-q),\z_S) =\dms(p_!u_!i_!\z_X[-r](-q),\z_S)\cong \dms(u_!i_!\z_X,p^!\z_S(q)[r])\\
 \cong \dms(u_!i_!\z_X,p^!\z_S(q)[r])\cong\dm(P)(u_!i_!\z_X,\z_P(q+d)[r+2d])
 \\
 \cong \dm(Y)(i_!\z_X,u^!\z_P(q+d)[r+2d])\cong
 \dm(Y)(i_!\z_X, \z_Y(q+d)[r+2d]).\end{gathered}$$
 
 Now, by Lemma \ref{lega43} we can assume  
 that $i$ is the filtered projective limit of closed embeddings $i_\be: U_\be \to Y_\be$ of smooth varieties over a perfect field (whereas the connecting morphisms of the system are smooth affine dominant).
Proposition \ref{pcisdeg}(\ref{iexch}, \ref{icont}) yields that it suffices to prove the vanishing 
 of $\dm(Y_\be)(i_{\be !}\z_{X_\be}, \z_{Y_\be}(q+d)[r+2d])$ for all $\be \in B$.

 Example 11.2.3 of ibid. easily yields that the latter group is isomorphic to the motivic cohomology group $H^{r+2d-2c,q+d-c}_{\cal M}(X_\be)$. It vanishes  since $r+2d-2c>2q+2d-2c$ (where $c$ is the codimension of $X_\be$ in $Y_\be$; see the notation and the calculations in loc. cit.).

\end{proof}

\subsection{Weight structures: short reminder and a new existence lemma}\label{sws}

\begin{defi}\label{dwstr}

I A pair of subclasses $\cu_{w\le 0},\cu_{w\ge 0}\subset\obj \cu$ 
will be said to define a weight
structure $w$ for $\cu$ if 
they  satisfy the following conditions:

(i) $\cu_{w\ge 0},\cu_{w\le 0}$ are additive and Karoubi-closed in $\cu$
(i.e. contain all $\cu$-retracts of their objects).

(ii) {\bf Semi-invariance with respect to translations.}

$\cu_{w\le 0}\subset \cu_{w\le 0}[1]$, $\cu_{w\ge 0}[1]\subset
\cu_{w\ge 0}$.

(iii) {\bf Orthogonality.}

$\cu_{w\le 0}\perp \cu_{w\ge 0}[1]$.

(iv) {\bf Weight decompositions}.

 For any $M\in\obj \cu$ there
exists a distinguished triangle
\begin{equation}\label{wd}
B\to M\to A\stackrel{f}{\to} B[1]
\end{equation} 
such that $A\in \cu_{w\ge 0}[1],\  B\in \cu_{w\le 0}$.

II The category $\hw\subset \cu$ whose objects are
$\cu_{w=0}=\cu_{w\ge 0}\cap \cu_{w\le 0}$, $\hw(Z,T)=\cu(Z,T)$ for
$Z,T\in \cu_{w=0}$,
 will be called the {\it heart} of 
$w$.


III $\cu_{w\ge i}$ (resp. $\cu_{w\le i}$, resp.
$\cu_{w= i}$) will denote $\cu_{w\ge
0}[i]$ (resp. $\cu_{w\le 0}[i]$, resp. $\cu_{w= 0}[i]$).

IV We denote $\cu_{w\ge i}\cap \cu_{w\le j}$
by $\cu_{[i,j]}$ (so it equals $\ns$ for $i>j$).

V We will  
call
$\cu^b=\cup_{i\in \z} \cu_{w\le i}\cap \cup_{i\in \z} \cu_{w\ge i}$ the class of {\it bounded} 
objects of $\cu$. We will say that $w$ is bounded if $\cu^b=\obj \cu$.

Besides, we will call $\cup_{i\in \z} \cu_{w\le i}$ the class of {\it bounded above} 
objects. 

VI $w$ will be called 
{\it non-degenerate from above} if $\cap_l \cu^{w\ge l}=\ns.$

VII Let $\cu$ and $\cu'$ 
be triangulated categories endowed with
weight structures $w$ and
 $w'$, respectively; let $F:\cu\to \cu'$ be an exact functor.

$F$ will be called {\it left weight-exact} 
(with respect to $w,w'$) if it maps
$\cu_{w\le 0}$ to $\cu'_{w'\le 0}$; it will be called {\it right weight-exact} if it
maps $\cu_{w\ge 0}$ to $\cu'_{w'\ge 0}$. $F$ is called {\it weight-exact}
if it is both left 
and right weight-exact.

 VIII We call a category $\frac A B$ a {\it factor} of an additive
category $A$
by its (full) additive subcategory $B$ if $\obj \bl \frac A B\br=\obj
A$ and $(\frac A B)(M,N)= A(M,N)/(\sum_{O\in \obj B} A(O,N) \circ
A(M,O))$.

\end{defi}

\begin{rema}\label{rstws}

1. A weight decomposition (of any $M\in \obj\cu$) is (almost) never canonical;
still we will sometimes denote (any choice of) a pair $(B,A)$ coming from in (\ref{wd}) by $(w_{\le 0}M,w_{\ge 1}M)$. 
For an $l\in \z$ we denote by $w_{\le l}M$ (resp. $w_{\ge l}M$) a choice of  $w_{\le 0}(M[-l])[l]$ (resp. of $w_{\ge 1}(M[1-l])[l-1]$).

 We will call (any choices of) $(w_{\le l}M,w_{\ge l}M)$  {\it weight truncations} of $M$.

2. A  simple (and yet  useful) example of a weight structure comes from the stupid
filtration on the homotopy categories of cohomological complexes
$K(B)$ for an arbitrary additive  $B$. 
In this case
$K(B)_{w\le 0}$ (resp. $K(B)_{w\ge 0}$) will be the class of complexes that are
homotopy equivalent to complexes
 concentrated in degrees $\ge 0$ (resp. $\le 0$).  The heart of this weight structure 
is the Karoubi-closure  of $B$ in $K(B)$.
 in the corresponding category.  

3. In the current paper we use the 'homological convention' for weight structures; 
it was previously used in \cite{hebpo}, \cite{wildic},  and  \cite{brelmot}, whereas in 
\cite{bws} and in \cite{bger} the 'cohomological convention' was used. In the latter convention 
the roles of $\cu_{w\le 0}$ and $\cu_{w\ge 0}$ are interchanged i.e. one considers   $\cu^{w\le 0}=\cu_{w\ge 0}$ and $\cu^{w\ge 0}=\cu_{w\le 0}$. So,  a complex $X\in \obj K(B)$ whose only non-zero term is the fifth one 
 has weight $-5$ in the homological convention, and has weight $5$ in the cohomological convention. Thus the conventions differ by 'signs of weights'; 
 $K(B)_{[i,j]}$ is the class of retracts of complexes concentrated in degrees $[-j,-i]$. 
  
 
\end{rema}

Now we recall those properties of weight structures that
will be needed below (and that can be easily formulated).
We will not mention more complicated matters (weight spectral sequences and weight complexes) 
here; instead we will just formulate
the corresponding 'motivic' results below.

\begin{pr} \label{pbw}
Let $\cu$ be a triangulated category. 

\begin{enumerate}

\item\label{iext} 
Let $w$  be a weight structure for
 $\cu$. Then  $\cu_{w\le 0}$, $\cu_{w\ge 0}$, and $\cu_{w=0}$
are extension-closed (see  Definition \ref{dses}(1)).

Besides, for any $M\in \cu_{w\le 0}$ we have $w_{\ge 0}M\in \cu_{w=0}$ (for any choice of $w_{\ge 0}M$).

\item\label{iwefun} For $\cu,w$ as in the previous assertions weight decompositions are {\it weakly functorial} i.e. any $\cu$-morphism of objects has a (non-unique) extension to a morphism of (any choices of) their weight decomposition triangles.

\item \label{iadjco}
A composition of left (resp. right) weight-exact functors is left (resp. right) weight-exact.

\item \label{iadj}

Let $\cu$ and $\du$ be triangulated categories endowed with weight structures $w$ and $v$, respectively. Let
$F: \cu \leftrightarrows \du:G$ be adjoint functors. Then $F$ is left weight-exact whenever $G$ is right weight-exact.

\item\label{iloc}

Let $w$  be a weight structure for
 $\cu$; let $\du\subset \cu$ be a 
triangulated subcategory of
$\cu$. Suppose
 that $w$ yields a weight structure $w_{\du}$ for  $\du$
(i.e. $\obj \du\cap \cu_{w\le
 0}$ and $\obj \du\cap \cu_{w\ge
 0}$ give a weight structure for $\du$). 

 Then $w$ also induces a weight structure on
 $\cu/\du$ (the localization i.e. the Verdier quotient of $\cu$
by the Karoubi-closure of $\du$) in the following sense: the Karoubi-closures of $\cu_{w\le
 0}$ and $\cu_{w\ge
 0}$ (considered as classes of objects of $\cu/\du$) give a weight structure $w'$ for $\cu/\du$
(note that $\obj \cu=\obj \cu/\du$). Besides, there exists a full embedding $\frac {\hw}{\hw_{\du}}\to \hw'$; $\hw'$ is the Karoubi-closure  of $\frac{\hw}{\hw_{\du}}$ in
$\cu/\du$.

Moreover, assume that  
the embedding of $i:\du\to \cu$ possesses both a left and a right adjoint i.e. that  $\du\stackrel{i_*}{\to}\cu\stackrel{j^*}{\to}\cu/\du$ is a part of gluing data (so that there also exist $j_*$ and $j^!$ and these six functors satisfy the conditions of Proposition \ref{pcisdeg}(\ref{iglup}, \ref{iglu}); see Chapter 9 of \cite{neebook}). Then all objects of  $\hw'$ come from $\hw$ (i.e. we do not need a Karoubi-closure here). 

\item\label{igluws}
Let $\du\stackrel{i_*}{\to}\cu\stackrel{j^*}{\to}\eu$ be a part of gluing data (as described in the previous assertion). 

Then for any pair of weight structures on $\du$ and $\eu$ (we will denote them by $w_\du$ and $w_\eu$, respectively)
there exists a 
 weight structure $w$ for $\cu$ such that both $i_*$ and $j^*$ are weight-exact (with respect to the corresponding weight structures). Besides, $i^!$ and $j_*$ are right weight-exact (with respect to the corresponding weight structures); $i^*$ and $j_!$ are left weight-exact. Moreover, 
$\cu_{w\ge 0}=C_1=\{M\in \obj
\cu:\ i^!(M)\in \du_{w_\du\ge 0} ,\ j^*(M)\in \eu_{w_\eu\ge 0} \}$ and
$\cu_{w\le 0}=C_2=\{M\in \obj \cu:\ i^*(M)\in \du_{w_\du\le 0} ,\ j^*(M)\in
\eu_{w_\eu\le 0} \}$. Lastly, $C_1$ (resp. $C_2$) is the  envelope  of $j_!(\eu_{w\le
0})\cup  i_*(\du_{w\le 0})$  (resp. of $ j_*(\eu_{w\ge 0})\cup i_*(\du_{w\ge 0})$).

\item\label{igluwsn} In the setting of  assertion \ref{igluws}, the weight structure $w$ described is the only weight structure for $\cu$ such that the functors $i_!,j_!,j^*$ and $i^*$ are left weight-exact (we will say that $w$ is {\it glued from} $w_\du$ and $w_\eu$). 

\end{enumerate}
\end{pr}
\begin{proof} 

Most of the assertions 
were proved in \cite{bws} (pay attention to Remark \ref{rstws}(3)!); some more precise information can be found in (the proof of) Proposition 1.2.3  of \cite{brelmot}.

It only remains to note that the 'moreover part' of assertion \ref{iloc}
generalizes Theorem 1.7 of \cite{wildic}. The proof of loc. cit. carries over to our (abstract) setting without any 
changes (note here that the existence of the natural transformation $j_!\to j_*$ used there is an immediate consequence of the adjunctions given by the gluing data setting).
\end{proof}

\begin{rema}\label{rlift}

Part \ref{iloc} of the proposition can be re-formulated as follows. If $i_*:\du\to \cu$ is an embedding of triangulated categories that is weight-exact (with respect to certain weight structures for $\du$ and $\cu$), an exact functor $j^*:\cu\to \eu$ is equivalent to the localization of $\cu$ by $i_*(\du)$, then there exists a unique weight structure $w'$ for $\eu$ such that $j^*$ is weight-exact; $\hw_{\eu}$ is the Karoubi-closure of $\frac {\hw} {i_*(\hw_{\du})}$ (with respect to the natural functor $\frac {\hw}{i_*(\hw_{\du})}\to \eu$).

\end{rema}

Now we prove a certain 
statement on the existence of weight structures (as none of the existence results  stated in \S4 of \cite{bws} 
are sufficient for our purposes); it 
is a slight reformulaion of the main result of \cite{paukcomp}.

\begin{pr}\label{pnews}

Let $\cu$ be triangulated category that is closed with respect to all coproducts;
let $C'\subset \cu$ be a (proper!) set of compact objects such that $C'\subset C'[1]$. Then the classes $C_1=C'^{\perp}[-1]$ and 
$C_2$ being the big extension-closure of $C'$ yield a weight structure for $\cu$ (i.e. there exists a $w$ such that $\cu_{w\ge 0}=C_1$, $\cu_{w\le 0}=C_2$).
\end{pr}
\begin{proof}
Obviously, $(C_1,C_2)$ are Karoubi-closed in $\cu$, $C_1[1]\subset C_1$, $C_2\subset C_2[1]$.
Besides, $C_2\perp C_1[1]$ by Lemma \ref{lbes}(\ref{isesperp}).

It remains to verify that any $M\in \obj \cu$ possesses a weight decomposition with respect to $(C_1,C_2)$. We apply (a certain modification of) the method used in the proof of Theorem 4.5.2(I) of \cite{bws} (cf. also the construction of  crude cellular towers in \S I.3.2 of \cite{marg}). 

  We construct a certain sequence of $M_k$ for $k\ge 0$ by induction in $k$ starting
from $M_0=M$. Suppose that $M_k$ (for some $k\ge 0$) is already constructed; then we take $P_k=\coprod_{(c,f):\,c\in C',f\in \cu(c,M_k)}c$; $M_{k+1}$ is a cone of the morphism $\coprod_{(c,f):\,c\in C',f\in \cu(c,M_k)}f:P_k\to M_k$.

Now we 'assemble' $P_k$. 
The compositions of the morphisms $h_k:M_{k}\to M_{k+1}$ given by this construction yields morphisms $g_i:M\to M_i$ for all $i\ge 0$. Besides, the octahedral axiom of triangulated categories 
immediately yields 
$\co (h_k)\cong P_k[1]$. Now we complete $g_k$ to distinguished triangles $B_k\stackrel{b_k}{\to}M \stackrel{g_k}{\to}M_k$. The octahedral axiom yields the existence of morphisms $s_i:B_i\to B_{i+1}$ that are compatible with $b_k$  
 such that 
$\co (s_i)\cong P_i$ for all $i\ge 0$.

We consider $B=\inli B_k$; by Lemma \ref{coulim}(1) $(b_k)$ lift to a certain morphism $g: B\to M$.
 We complete $b$ to a distinguished triangle $B\stackrel{b}{\to} M\stackrel{a}{\to} A\stackrel{f}{\to} B[1]$.
This triangle will be our candidate for a weight decomposition of $M$.

First we note that $B_0=0$; since $\co (s_i)\cong P_i$ we have
 $B\in C_2$ by the definition of the latter.

It remains to prove that $A\in C_1[1]$ i.e. that $C'\perp A$. 
For a $c\in C'$ we should check that
$\cu(c,A)=\ns$. The long exact sequence  $$\dots  \to \cu(c,B)\to  \cu(c,M)\to \cu(c,A)\to \cu(c, B[1])\to \cu(c,M[1])\to\dots $$
translates this into: $\cu(c,-)(b)$ is surjective and $\cu(c,-)(b[1])$ is injective. 
Now, by Lemma \ref{coulim}(3), $\cu(c,B)\cong\inli \cu(c,B_i)$ and $\cu(c,B[1])\cong\inli \cu(c,B_i[1])$. Hence the long exact sequences
$$\dots  \to \cu(c,B_k)\to  \cu(c,M)\to \cu(c,M_k)\to \cu(c, B[1])\to \cu(c,M_k[1])\to\dots $$ yield: it suffices to verify that 
$\inli \cu(c,M_k) =\ns$ (note here that $h_k$ are compatible with $s_k$). 
Lastly, 
 $\cu(c,P_k)$ surjects onto $\cu(c,M_k)$; hence 
the group $\cu(c,M_k)$ dies in $\cu(c,M_{k+1})$ for any $k\ge 0$ and we obtain the result. 


\end{proof}

\begin{rema}
Much interesting information on $w$ can be obtained from (the dual to) Theorem 2.2.6 of \cite{bger}. 
\end{rema}

\section{Our main results: the construction and  properties of the Chow weight structure}\label{swchow}

In \S\ref{schowunr}  we 
define a certain weight structure $\wchow$ for $\dms$. 
and study its properties. In particular, we describe the 'functoriality' of $\wchow$ (with respect to functors of the type $f^*,f_*,f^!$, and $f_!$, for $f$ being a quasi-projective morphism of schemes). 

In \S\ref{srat} we discuss the relation of our 'integral weights' with the ones for motives with rational coefficients.

In \S\ref{sappl} we briefly describe some 
(immediate) consequences 
of our results.

\subsection{The main results 
}\label{schowunr}


\begin{defi}\label{dwchow}
For a scheme $S$ we denote by $C'=C'(S)$ the set of all $p_!(\z_P)(r)[2r-q]$, where  $p:P\to Z$ runs through all quasi-projective 
 morphisms with regular domain, $r\in \z$, $q\ge 0$. Note that we can assume $C'$ to be a (proper) set; its objects are compact in $\dms$. 

Then we denote by $\wchow=\wchow(S)$ the weight structure corresponding to $S$ by Proposition \ref{pnews} (i.e. $\dms_{\wchow\ge 0}=C'^\perp[-1]$, $\dms_{\wchow\le 0}$ is the big extension-closure of $C'$). 

\end{defi}

We prove the main properties of $\wchow$.

\begin{theo}\label{twchowa}

I The functor $-(n)[2n](=\otimes \z(n)[2n]):\dms\to\dms$ is weight-exact with respect to $\wchow$ for any 
$S$ and any $n\in \z$.

II Let $f:X\to Y$ be a 
quasi-projective  
 morphism of 
 schemes.

1. $f_!$ is left weight-exact, $f^!$ is right weight-exact. Moreover, $f^*$ is left weight-exact and  $f_*$ is right weight-exact if $X$ and $Y$ are regular.

2. Suppose moreover that $f$ is smooth. Then $f^*$
and $f^!$ are also weight-exact.

3. Moreover, $f^*$ is weight-exact for  $f$ being  a
(filtered) projective limit of smooth quasi-projective 
 morphisms such that the corresponding connecting morphisms are dominant smooth 
 affine. 
 

III  Let $i:Z\to X$ be a closed immersion of regular schemes; 
let $j:U\to X$ be the complimentary open immersion.

1. $\chow(U)$ is 
the factor (in the sense of  Definition \ref{dwstr}(VIII)) of $\chow(X)$ by $i_*(\chow(Z))$. 

2. For $M\in \obj \dmx$ we have: $M\in \dmx_{\wchow\ge 0}$ (resp. $M\in \dmx_{\wchow\le 0}$) whenever $j^!(M)\in \dm(U)_{\wchow\ge 0}$ and $i^!(M)\in \dm(Z)_{\wchow\ge 0}$ (resp.  $j^*(M)\in \dm(U)_{\wchow\le 0}$ and $i^*(M)\in \dm(Z)_{\wchow\le 0}$).

IV Let $S=\cup S_l^\al$ be a very regular stratification (of a regular $S$), $j_l:S_l^\al\to S$ are the corresponding immersions. Then for $M\in \obj \dms$ we have: $M\in \dms_{\wchow\ge 0}$ (resp. $M\in \dms_{\wchow\le 0}$) whenever $j_l^!(M)\in \dm(S_l^\al)_{\wchow\ge 0}$ (resp. $j_l^*(M)\in \dm(S_l^\al)_{\wchow\le 0}$)  for all $l$.

V 
For a regular $S$ the following statements are valid.

1. Any object of $\dmcs$ is bounded above with respect to $\wchow(S)$.

2. $\wchow(S)$ is non-degenerate from above.

3.  $\z_S\in \dms_{\wchow\le 0}$.

4. If $S$ is also pro-smooth, then $\z_S\in \dms_{\wchow=0}$.


\end{theo}
\begin{proof}

I By Proposition \ref{pbw}(\ref{iadj}) it suffices to verify that  all $-(n)[2n]$ are left weight-exact. The latter is immediate from the definition of $\dms_{\wchow\le 0}$ by Lemma \ref{lbes}(\ref{isesfun}).



II1. The statement for $f_!$ is obvious (here we apply  Lemma \ref{lbes}(\ref{isesfun} again). Next, Proposition \ref{pbw}(\ref{iadj}) yields the statement for $f^!$.

In order to verify the remaining parts of the assertion it suffices to consider the case when $f$ is either smooth, or is a closed embedding of regular schemes.

If $f$ is smooth then Proposition \ref{pcisdeg}(\ref{iexch}) (together with Lemma \ref{lbes}(\ref{isesfun}) yields that $f^*(\dm(Y)_{\wchow\le 0})\subset \dm(X)_{\wchow\le 0}$. Applying Proposition \ref{pbw}(\ref{iadj}) we also obtain the right weight-exactness of $f_*$.

By loc. cit. we obtain: it remains to verify the left weight-exactness of $f^*$ when $f$ is a closed embedding. Denote by $j:U\to Y$ the open embedding complimentary to $f$.
It suffices to check (after making obvious reductions) that for $M=p_!(\z_P)$ where  $P:P\to Y$ is  smooth quasi-projective,  we have:
$f^*M$ can be obtained from objects of the form $q_{i!}(\z_{Q_i})$ for some quasi-projective $q_i:Q_i\to X$, $Q_i$ are regular,
by 'extensions' (cf. Lemma \ref{lbes}(\ref{iseses}).
Now, choose a very regular stratification $\al$ of $P$ each of whose components is mapped by $p$ either to $X$ or to $U$. 
By Lemma \ref{l4onepo}(1), it suffices to verify that $i^*p_!j_{l!}^\al(\z_{P^l_\al})\in C'(X)$. Now, if $p\circ j_l$ factorizes through $U$, then $f^*p_!j_{l!}^\al(\z_P)$ factorizes through $f^*j_!=0$ (see Proposition \ref{pcisdeg}(\ref{iglup})). On the other hand, if  $p\circ j_l=f\circ p'$ (for some $p':P^l_\al\to X$) then $f^*p_!j_{l!}^\al(\z_{P^l_\al})\cong f^*f_!p'_!
(\z_{P^l_\al})\cong 
j_{l!}(\z_{P^l_\al})\in C'(X)$.

2. Obviously, 
we can assume that $f$ is equi-dimensional of relative dimension $s$. Then
$f^!(-)\cong f^*(-)(s)[2s]$ by Proposition \ref{pcisdeg}(\ref{ipur}).  Hence both the left and the right hand side
are weight-exact by the combination of the previous assertions of our theorem (note also that we actually verified the left weight-exactness of $f^*$ for an arbitrary smooth $f$ in the proof of the previous assertion).

3. Passing to the limit (using 
Remark \ref{ridmot}) we prove 
the left weight-exactness of $f^*$. 

In order to verify its right weight-exactness (by the definition of $\wchow$) for $X=\prli_\be X_\be$ ($\be\in B$), $Y_0=X$,
$p_{\be}:X_\be \to X$, and $p_{\be,0}: X_\be \to X_0$, we should check: for any $O\in C'(X)$,
$N\in \dm(Y)_{\wchow\ge 1}$,  we have 
$f^*M\perp O$. Since all the $p_{\be,0}^*$ are (right) weight-exact, we can replace $X$ by any $X_\be$ in this statement; hence we may assume that $O=f^*N$ for some $N\in C'(Y)$ (by Lemma \ref{lega43}).
Then it remains to 
combine the previous assertion with 
Proposition \ref{pcisdeg}(\ref{icontp}).

III Since $i_*\cong i_!$ in this case, $i_*$ is weight-exact by assertion II1. $j^*$ is weight-exact by assertion II2.

1. $\dm(U)$ is the localization of $\dm(X)$ by $i_*(\dm(Z))$
by  Proposition \ref{pcisdeg}(\ref{iglu}). Hence   Proposition \ref{pbw}(\ref{iloc}) yields the result (see Remark \ref{rlift}).

2.  Proposition \ref{pcisdeg}(\ref{iglu}) 
yields: $\wchow(X)$ is exactly the weight structure obtained by 'gluing $\wchow(Z)$ with $\wchow(U)$' via   Proposition \ref{pbw}(\ref{igluws}) (here we use part \ref{igluwsn} of loc. cit.). Hence loc. cit. yields the result (note that $j^*=j^!$). 


IV The assertion can be easily proved by induction on the number of strata  using assertion III2.


V1. Immediate from Proposition \ref{pcisdeg}(\ref{icompgen}). 

2. Loc.cit. also yields that for any
$N\in \obj \dms$ there exists a non-zero morphism in $\dms(O[i],N)$ for some $i\in \z$, $O\in C'(S)$. This is equivalent to assertion V2 (by the definition of $\wchow$).

3.  We have $\z_S\in C'(S)$; hence $\z_S\in \dms_{\wchow\le 0}$. 

4. Let $S\cong \prli S_\be$ (as in the definition of pro-smooth schemes). 
Assertion II3 (along with Proposition \ref{pcisdeg}(\ref{iupstar})) yields: it suffices to verify the statement in question for one of $S_\be$. Hence we may assume that $S$ is a smooth quasi-projective variety over a perfect field $K$. Moreover, by assertion II2 it suffices to consider the case $S=\spe K$. 
In this case the statement is immediate from Lemma \ref{l4onepo}(2).



\end{proof}

\begin{rema}\label{rexplwd}

1. We do not know whether a (general) compact motif always possesses a 'compact weight decomposition', though this is always true for motives with rational coefficients (by the central results of \cite{hebpo} and \cite{brelmot}; cf. \S\ref{srat} below).
This makes the search of 'explicit' weight decompositions (even more) important. Note that the latter allow the calculation of {\it weight filtrations} and {\it weight spectral sequences} for cohomology; see Proposition \ref{pwss} below.
The problem here is that our results do not provide us with 'enough' compact objects of $\dms_{\wchow\ge 0}$; they 
do not yield
any   non-$\q$-linear objects of $\dms_{\wchow\ge 0}$ (see \S\ref{srat})   at all unless $S$ is a scheme over a field.

2. Still we describe an important case when an explicit weight decomposition is known.

Adopt the notation and assumptions  of Proposition \ref{pcisdeg}(\ref{iglu}). Let $p_U:P_U\to U$ be a projective morphism such that $P_U$ is pro-smooth; denote $p_{U!}\z_{P_U}$ by $N$. Then we have $ j_!(N[1])\in \dmx_{\wchow}\le 1$; hence $w_{\le 0}j_!(N[1])\in \dmx_{\wchow=0}$.

Now assume that $P_U$ possesses a pro-smooth $X$-model i.e. that $P_U=P\times_X U$ for a projective $p:P\to X$, $P$ is pro-smooth. Then for $M=p_!\z_P$ we have: $N=j^*(M)(=j^!(M))$.  Since  $M\in \dmx_{\wchow=0}$, 
 the distinguished triangle $$ i_*i^*(M)\to j_!(N)[1](=j_!j^!(M)[1]) \to M[1]$$ (given by Proposition \ref{pcisdeg}(\ref{iglu})) yields a $\wchow(X)$-weight decomposition of
$j_!(N)[1]$ (note that both of its components are compact). 
Hence 
weight decompositions relate $P_U$ with $P$ when the latter exists;
still they exist and are 'weakly functorial' (see Proposition \ref{pbw}(\ref{iwefun})) in  $N$ in the general case also. 
 Note that such an $X$-model 
for $P_U$ always exists if $X$ is a variety over a characteristic $0$ field (by Hironaka's resolution of singularities); hence our methods yield a certain substitute for the resolution of singularities in a more general situation. This can be applied to the study of motivic cohomology of $P_U$ with integral coefficients via weight filtration (see Proposition \ref{pwss}(II1) below) for the 
cohomology of $j_!N$; see also Proposition 3.3.3 of \cite{brelmot} for more detail (in the setting of motives with rational coefficients).

3. If $S$ is a variety over a field (or a pro-smooth scheme) 
there is a subcategory of $\dms$ (that also lies in $\dmcs$) such that our $\wchow(S)$ restricts to a 'very explicit' weight structure for it. This is the category of 'smooth motives' considered in \cite{lesm}; by Corollary 6.14 of ibid. it equals the subcategory of $\dms$ generated by 'homological motives' of smooth projective $P/S$ (and the heart of this restriction is given by retracts of 
these $M_S(P)$.
One can also restrict $\wchow$ to the subcategory of Tate motives inside these smooth ones; see Corollary 6.16 of ibid.

\end{rema}

Now we prove for a regular $S$ that positivity 
of objects of $\dms$ (with respect to $\wchow$) 
can be 'checked at points'. Also, 
 'weights of compact objects are lower semi-continuous'.

\begin{pr}\label{ppoints}


Let  $M\in \obj \dms$. Denote by  $\sss$   the set of 
(Zariski) points of $S$. For a $K\in \sss$ we will denote the corresponding morphism $K\to S$ by $j_K$.

1. Let $S$ be regular. Then $M\in \dms_{\wchow\ge 0}$ 
if and only if for any $K\in \sss$ we have $j_K^!(M)\in \dm(K)_{\wchow\ge 0}$; 

2. Let $K$ be a generic point of  $S$, $M\in \obj\dmcs$. 
 Suppose that 
 $j_K^*(M)\in \dmk_{\wchow\le 0}$ 
 Then there exists an open immersion $j:U\to S$, $K\in U$, such that $j^*(M)\in \dm(U)_{\wchow\le 0}$. 

\end{pr}
\begin{proof}

1. Combining parts II1 and II3 of  Theorem \ref{twchowa} we obtain: if $M\in \dms_{\wchow\ge 0}$  then $j_K^!(M)\in \dm(K)_{\wchow\ge 0}$ for any $K\in \sss$.

Now we prove the converse implication.
 We prove it via certain 
 noetherian induction: we suppose that our assumption is true for motives over any 
 regular subscheme of $S$ that is not dense in it.
 
 Let $M\in \obj\dms$ satisfy  $j_K^!(M)\in \dm(K)_{\wchow\ge 0}$  for any  $K\in \sss$. We should check that $N\perp M$ for any $N\in C'(S)[-1]$. So, we choose some $g\in \dms(N,M)$ and prove that $g=0$.
 
We choose a generic point $K$ of $S$. By part II3 of  Theorem \ref{twchowa} we have $j_K^*(N)\in \dm(K)_{\wchow\le -1}$; since $j_K^*=j_K^!$ we also have $j_K^*(N)\in \dm(K)_{\wchow\ge 0}$. Hence $j_{K}^*(g)=0$. By Proposition \ref{pcisdeg}(\ref{icontp})  there exists an open immersion $j:U\to S$ ($K\in U$)
such that $j^*(g)=0$. We choose a very regular stratification $\al$ of $S$ such that $U=S_{l_0}^{\al}$ is one of its components. By Lemma \ref{l4onepo}(1) it suffices to verify that $j_{l!}j_l^*(N)\perp M$ for any $l\in L$.
Now, we have $\dms(j_{l!}j_l^*(N),M)\cong \dm(S_{l}^{\al})(j_l^*N,j_l^!M)$. By the inductive assumption we have 
$\dm(S_{l}^{\al})(j_l^*N,j_l^!M)=\ns$ for any $l\neq l_0$ (since $N\in \dm(S_{l}^{\al})_{\wchow\le -1}$), whereas $\dm(S_{l}^{\al})(j_l^*N,j_l^!M)=\dm(S_{l}^{\al})(j_l^*N,j_l^*M)=\ns$.

2. We consider a weight decomposition of $M$: $B\to M
\stackrel{g}{\to} A\to B[1]{\to} M[1]$.
We obtain that $j^*_K(g)=0$ (since $j^*_K(M)\perp \dmk_{\wchow\ge 0}[1]$). Again,  by Proposition \ref{pcisdeg}(\ref{icontp}) we obtain:  there exists an open immersion $j:U\to S$ ($K\in U$)
such that $j^*(g)=0$. Hence 
$j^*(M)$ is a retract of $j^*(B)$. Since $j^*(B)\in \dm(U)_{\wchow\le 0}$ 
and $\dm(U)_{\wchow\ge 0}$ is Karoubi-closed in $\dm(U)$, we obtain the result.

\end{proof}

\subsection{
 Comparison with 
 weights for Beilinson motives}\label{srat}

In \cite{degcis} certain categories of $\dmlas$ were constructed and studied for any coefficient ring $\Lambda
\subset \q$ ($1\in \Lambda$; actually any commutative ring with a unit is possible). In this paragraph we will consider only the categories of the type $\dmqs$ (along with $\dms$) since their properties are better understood. Yet note: it also makes some sense to invert only the positive residue field characteristics of the corresponding base fields (in $\Lambda$); see the second assertion in Proposition 11.1.5 of ibid and \cite{bzp}.

The properties of $\dmcqs\subset \dmqs$ were stated in  (part C of) the Introduction of ibid.; see also \S2 of \cite{hebpo} and \S1.1 of  \cite{brelmot}. Here we only note that all the results on $\dm(-)$ that were stated and proved above also hold for motives with rational coefficients. Besides, for any finite type $f$ all the corresponding motivic image functors (i.e. $f^*$, $f_*$, $f_!$, and $f^!$) respect compactness of motives. Moreover, the analogue of Lemma \ref{l4onepo}(2) holds for an any regular $S$ (i.e. pro-smoothness is not needed); Proposition \ref{pcisdeg}(\ref{iglu}) holds for arbitrary closed embeddings. This allowed to construct in \S2.3 of \cite{brelmot} a certain bounded Chow weight structure for $\dmcqs$ that extends to $\dmqs$. 
Combining Proposition 2.3.4(I2) and Proposition 2.3.5 of ibid. one can easily prove that this weight structure can be described similarly to Definition \ref{dwchow} above (this was our motivation for giving such a definition). Note also: in the case when $S$ is {\it reasonable} (i.e. if there exists an excellent 
separated
scheme $S_0$ of dimension lesser than or equal to $2$ 
such that $S$ is 
of finite type over $S_0$) then this weight structure also has another more simple description  (in terms of certain Chow motives over $S$ that yield its heart) given by Theorem 3.3 of \cite{hebpo} and Theorem 2.1.1 of \cite{brelmot}. So, we know much more on weights for motives with rational coefficients than for the ones with integral ones.

Yet the results of the current can be somewhat useful for the study of motives and cohomology with integral coefficients. We note that we have natural comparison functors $\dm(-)\otimes \q\to \dmq(-)$; they commute with all functors of the type $f^*$, and are isomorphisms (and commute with 'everything else') when restricted to regular base schemes (see Proposition 11.1.5 of \cite{degcis}).  It easily follows that for any regular $S$ the comparison isomorphism $\dm(S)\otimes \q\to \dmq(S)$ is weight-exact (here we take $((\dm(S)\otimes \q)_{\wchow_\le 0}, (\dm(S)\otimes \q)_{\wchow_\ge 0})$ being the images of $(\dm(S)_{\wchow_\le 0}, \dm(S)_{\wchow_\ge 0})$; we do not need Karoubizations). 
We obtain (in particular) that the weight filtrations and weight spectral sequences (see the next paragraph) for any (co)homology theory with rational coefficients defined via our 'current' $\wchow$ coincide with the ones defined using the 'rational version' (whereas the latter is certainly easier to calculate).



\subsection{Applications 
}\label{sappl}

First we recall that the embedding $\hwchows\to K(\hwchows)$ factorizes 
 through a certain exact {\it weight complex} functor
$t_S:\dms\to K(\hwchows)$ (similarly to 
 Proposition 5.3.3 of \cite{bws}, this follows from the existence of the Chow weight structure for $\dms$ along with the fact that it admits a differential graded enhancement; the latter property of $\dms$ can be easily verified since it 
 is defined in terms of certain  derived categories of sheaves over $S$).


Now we discuss (Chow)-weight spectral sequences and 
 filtrations for homology and cohomology of motives.  We note that any weight structure yields certain weight spectral sequences for any (co)homology theory. 

\begin{pr}\label{pwss}
Let $\au$ be an abelian category.

I Let $H:\dms\to \au$ be a homological functor; for any $r\in \z$ denote $H\circ [r]$ by $H_r$.

For an $M\in \obj\dms$ we denote by $(M^i)$ the terms of $t(M)$ (so $M^i\in \obj \hwchows$; here we can take any possible choice of $t(M)$). 


Then the following statements are valid.

1. There exists a ({\it Chow-weight})  spectral sequence $T=T(H,M)$ with $E_1^{pq}=
H_q(M^p)$; the differentials for $E_1T(H,M)$ come from $t(M)$. It converges to $H_{p+q}(M)$ if $M$ is bounded.

2. $T(H,M)$ is $\dms$-functorial in $M$ (and does not depend on any choices) starting from $E_2$.

II1. Let $H:\dms\to \au$ be any contravariant functor.
Then for any $m\in \z$ the object $(W^{m}H)(M)=\imm (H(\wchow_{\ge m}M)\to H(M))$
does not depend on the choice of $\wchow_{\ge m}M$; it is functorial in $M$.

We call the filtration of $H(M)$ by $(W^{m}H)(M)$ its {\it Chow-weight} filtration.

2. Let $H$ be cohomological.  For any $r\in \z$ denote $H\circ [-r]$ by $H^r$. 

Then the natural dualization of assertion I is valid.
For any $M\in \obj \dms$ we have a spectral sequence 
with $E_1^{pq}=
H^{q}(M^{-p})$; it converges to $H^{p+q}(M)$ if $M$ is bounded. 
Moreover, in this case the step of filtration given by ($E_{\infty}^{l,m-l}:$ $l\ge k$)
 on $H^{m}(M)$ equals $(W^k H^{m})(M)$ (for any $k,m\in \z$). $T$ is functorial in $H$ and $M$ starting from $E_2$.

\end{pr}

\begin{proof}

I Immediate from Theorem 2.3.2 of ibid. 

II1. This is  Proposition 2.1.2(2) of ibid.

2. Immediate from Theorem 2.4.2 of ibid.

\end{proof}

\begin{rema}\label{rintel}

1.  We obtain certain functorial {\it Chow-weight} spectral sequences and
filtrations for any (co)homology of motives. In particular, we have them
for  \'etale and motivic (co)homology of motives.
Note that these results that cannot be proved using 'classical' (i.e. Deligne's) methods, since the latter heavily rely on the degeneration of (an analogue of) $T$ at $E_2$. 

2. $T(H,M)$ can be naturally described in terms of the {virtual $t$-truncations} of $H$
(starting from $E_2$); see 
\S2 of \cite{bger} and \S4.3 of \cite{brelmot}. 


\end{rema}

\end{document}